\documentclass{article}
\usepackage{amsmath,amssymb,enumerate,ascmac,mathrsfs,amsthm}
\usepackage{hyperref}

\newtheorem{definition}{Definition}[section]
\newtheorem{theorem}[definition]{Theorem}
\newtheorem{proposition}[definition]{Proposition}
\newtheorem{lemma}[definition]{Lemma}
\newtheorem{remark}[definition]{Remark}

\renewcommand{\leqslant}{\leq}
\renewcommand{\geqslant}{\geq}
\newcommand{\bN}{\mathbb{N}}
\newcommand{\bP}{\mathbb{P}}
\newcommand{\bR}{\mathbb{R}}
\newcommand{\bE}{\mathbb{E}}

\newcommand{\cF}{\mathcal{F}}
\newcommand{\cL}{\mathcal{L}}
\newcommand{\cS}{\mathcal{S}}
\newcommand{\al}{\alpha}
\newcommand{\be}{\beta}
\newcommand{\de}{\delta}
\newcommand{\ep}{\varepsilon}
\newcommand{\ga}{\gamma}
\newcommand{\he}{\theta}

\newcommand{\om}{\omega}
\newcommand{\Ga}{\Gamma}

\newcommand{\sgn}{\mathrm{sgn}}
\newcommand{\supp}{\mathrm{supp}}
\numberwithin{equation}{section}

\begin{document}
\title{\bf Tanaka formula for strictly stable processes}
\author{Hiroshi TSUKADA
\footnote{Graduate School of Science,
Osaka City University,
Japan.
e-mail: \texttt{hrstsukada@gmail.com}}}
\date{}
\maketitle
\begin{abstract}
For symmetric L\'evy processes, if the local times exist, the Tanaka formula has already constructed via the techniques in the potential theory by Salminen and Yor (2007).
In this paper, we study the Tanaka formula for arbitrary strictly stable processes with index $\al \in (1,2)$
including spectrally positive and negative cases in a framework of It\^o's stochastic calculus.
Our approach to the existence of local times for such processes is different from Bertoin (1996).
\end{abstract}

\section{Introduction}\label{sec1}
First, we begin with the definition of a local time for a stochastic process $X = (X_t)_{t\geqslant0}$.
\begin{definition}\rm{
A family of random variables 
$L = \{L^a_t: a \in \bR, t\geqslant 0\}$ is a local time of $X$
if, for any bounded Borel measurable function $f : \bR \to [0,\infty)$ and $t \geqslant0$,
\begin{equation*}
\int^t_0 f(X_s) ds = \int_\bR f(a) L^a_t da \qquad \text{a.s.},
\end{equation*}
which is called an occupation time formula.
}
\end{definition}

It is known that there exist local times of Brownian motions (cf. Berman \cite{Berm})
and necessary and sufficient conditions of existence for L\'evy processes can be found in Bertoin \cite{Bert}.
They considered the local time as the Radon--Nikodym derivative of the occupation time measure $\mu_t$ defined,
for any Borel measurable function $f : \bR \to [0,\infty)$ and $t \geqslant 0$,
by
\begin{equation*}
\int^t_0 f(X_s) ds = \int_\bR f(a) \mu_t(da).
\end{equation*}
The existence of almost surely jointly continuous local times
was studied by Trotter \cite{Tro} for Brownian motions,
and by Boylan \cite{Boy} for stable processes with index $\al \in (1,2)$,
and necessary and sufficient conditions for the almost sure joint continuity of local times were given by Barlow \cite{Bar}.

Let $B=(B_t)_{t\geqslant0}$ be a Brownian motion on $\bR$.
It is well known that the Tanaka formula for a Brownian motion holds:
\begin{equation*}
|B_t-a| - |B_0-a| = \int^t_0 \sgn(B_s-a) dB_s + L^a_t,
\end{equation*}
where $L^a$ denotes the local time of the Brownian motion at level $a$,
and for more details, see, e.g. Ikeda and Watanabe \cite{Ike}.
By L\'evy's characterization, $\int^\cdot_0 \sgn(B_s-a) dB_s$ is another Brownian motion.
By setting $a = 0$, we know the process $|B|$ is the reflection of the Brownian motion which means that $(|B|,L^0)$ is the solution of the Skorohod problem for the Brownian motion.
By the Doob--Meyer decomposition, 
the local time $L^a$ can be understood as the unique bounded variation process which is the difference of the positive submartingale $|B-a|$ and the martingale $\int^\cdot_0\sgn(B_s-a)dB_s+|B_0-a|$. 
Thus, in this paper we say that the Tanaka formula holds if the local time can be expressed as the difference of such processes because it has a gap between the Skorohod solution and local times in the case of jump processes.
The Tanaka formula has been already studied by Yamada \cite{Yam} for symmetric stable processes with index $\al\in(1,2)$, and by Salminen and Yor \cite{Sal} for symmetric L\'evy processes when local times exist.
On the other hand, Watanabe \cite{Wat} and Engelbert, Kurenok and Zalinescu \cite{Eng} focused on the solution of the Skorohod problem.

In \cite{Sal}, they constructed the Tanaka formula for symmetric L\'evy process $X$ by using the continuous resolvent density $u^p$:
\begin{equation*}
v(X_t-a)=v(a)+M^a_t+L^a_t
\end{equation*}
where $v(x) := \lim_{p \to 0}\left(u^p(0)-u^p(x)\right)$ which is called a renormalized zero resolvent, and $M^a$ is a martingale.
But the existence of the limit is not clear in the asymmetric case and the representation for the martingale part is not given.
In \cite{Yan} Yano obtained a renormalized zero resolvent for asymmetric L\'evy process under some conditions, and it associates the Tanaka formula based on local times with a harmonic function for the killed process upon hitting zero
because the local time of such processes at level zero becomes zero.

In this paper, we shall show the Tanaka formula for arbitrary strictly stable processes including spectrally and negative cases with index $\al \in (1,2)$ using It\^o's stochastic calculus.
Although our formula might give an insight on a reflection problem of jump type processes,
the martingale part in the formula is not suitable to become the same law of the original process.
Thus, there exists a gap between the Skorohod solution and the existence of local times in case of jump processes.

In Section \ref{sec2}, we shall prepare the It\^o formula, the infinitesimal generator and the moments about stable processes. 
In Section \ref{sec3}, we construct the Tanaka formula for strictly stable processes.

\section{Preliminaries}\label{sec2}
Let $C^n_{1+,b}(\bR)$ be the elements of $C^n(\bR)$ such that all derivatives of any orders greater than or equal to $1$ are bounded,
and $C^n_c(\bR)$ be the functions in $C^n(\bR)$ having compact support.
Let $\cS(\bR)$ be the Schwartz space of rapidly decreasing functions on $\bR$.

A L\'evy process $Y=(Y)_{t\geqslant0}$ is characterized by
the Levy--Khintchine representation given as
\begin{equation*}
\bE[e^{iuY_t}]=\exp\left[t\left(ibu-\frac{1}{2}au^2+\int_{\bR_0}\{ e^{iuy}-1-iuy\textbf{1}_{|y|\leqslant1} \} \nu(dy)\right)\right],
\end{equation*}
for a drift parameter $b\in\bR$, a Gaussian coefficient $a\geqslant0$ and a L\'evy measure $\nu$ on $\bR_0:=\bR\setminus\{0\}$ satisfying the following integrability condition:
\begin{align*}
\int_{\bR_0} (|h|^2\wedge1)\nu(dh)< \infty.
\end{align*}
Note that a L\'evy process is characterized by the triplet $(b,a,\nu)$.

In this paper,
we are concerned with one-dimensional strictly stable process $X=(X_t)_{t\geqslant0}$ with index $\al \in (1,2)$.
This process $X$ is characterized by the triplet $(b_\al,0,\nu_\al)$
where the L\'evy measure $\nu_\al$ is given by
\begin{equation*}
\nu_\al(dh) =
  \begin{cases}
  c_+|h|^{-\al-1}dh &\quad \text{on $(0,\infty)$}, \\
  c_-|h|^{-\al-1}dh &\quad \text{on $(-\infty,0)$},
  \end{cases}
\end{equation*}
where $c_+$ and $c_-$ are non-negative constants such that $c_+ + c_- > 0$,
and the drift parameter $b_\al$ is given by
\begin{align*}
b_\al=-\int_{|h|>1}h\nu_\al(dh)=-\frac{c_+-c_-}{\al-1}.
\end{align*}
The L\'evy--Khintchine representation for $X$ is given as
\begin{align*}
\bE [e^{iuX_t}] = e^{t\eta(u)}, \quad u\in\bR,
\end{align*}
where the function $\eta$ is the L\'evy symbol of $X$ given by
\begin{equation*}
\eta(u) = -d|u|^\al \left(1-i\be\sgn(u)\tan\left(\frac{\pi \al}{2}\right)\right),
\end{equation*}
where $d > 0$ and $\be \in [-1,1]$ are given by
\begin{equation*}
d=\frac{c_+ +c_-}{2c(\al)}, \qquad \be=\frac{c_+-c_-}{c_+ +c_-}
\end{equation*}
with
\begin{equation*}
c(\al)=\frac{1}{\pi}\Ga(\al+1)\sin\left(\frac{\pi \al}{2}\right)
\end{equation*}
and the left-continuous signum function denoted by
\begin{equation*}
\sgn(x) :=
  \begin{cases}
  1 &\quad \text{if $x > 0$}, \\
  -1 &\quad \text{if $x \leqslant 0$}.
  \end{cases}
\end{equation*}

By the L\'evy--It\^o decomposition (see, Applebaum \cite[Theorem 2.4.16]{App}),
$X$ can be represented by
\begin{align}\label{2-1}
X_t &= X_0 - \int_{|h|>1}h\nu_\al(dh)t + \int^t_0 \int_{|h|\leqslant1}h\tilde{N}(ds,dh) + \int^t_0 \int_{|h|>1}h N(ds,dh) \notag
\\&= X_0 +\int^t_0\int_{\bR_0}h\tilde{N}(ds,dh),
\end{align}
where $N$ is the jump measure of $X$ which is a Poisson random measure with the intensity measure $ds\nu_\al(dh)$ and the compensated measure $\tilde N$ is defined by
\begin{equation*}
\tilde N(ds,dh) := N(ds,dh) - ds\nu_\al(dh).
\end{equation*}
Since $\bE[|X_t|]<\infty$, it follows from \eqref{2-1} that $X$ is a martingale.
From the It\^o formula (\cite[Theorem 4.4.7]{App}),
we have the formula for \eqref{2-1}.

\begin{proposition}\label{prop2-1}
For each $f \in C^2_{1+,b}(\bR)$, the following formula holds:
\begin{align*}
f(X_t) - f(X_0)
&=-\int^t_0 \int_{|h|>1}f'(X_s)h \nu_\al(dh)ds
\\&\quad+\int^t_0 \int_{|h| > 1} \{f(X_{s-} + h) - f(X_{s-})\} N(ds,dh)
\\&\quad+\int^t_0 \int_{|h| \leqslant 1} \{f(X_{s-} + h) - f(X_{s-})\} \tilde{N}(ds,dh)
\\&\quad+\int^t_0 \int_{|h| \leqslant 1} \{f(X_s + h) - f(X_s) - f'(X_s)h\} \nu_\al(dh)ds.
\end{align*}
\end{proposition}

\begin{remark}
\rm{
Since $f \in C^2_{1+,b}(\bR)$,
it follows from the mean value theorem that
\begin{align*}
\int_{|h| \leqslant 1}|f(x+h) - f(x)|^2\nu_\al(dh)
&=\int_{|h| \leqslant 1}\left|\int^1_0f'(x+\he h)hd\he\right|^2\nu_\al(dh)
\\&\leqslant K^2\int_{|h| \leqslant 1} |h|^2\nu_\al(dh)
=\frac{K^2(c_++c_-)}{2-\al}.
\end{align*}
where $K$ is a positive constant such that $|f'(x)|\leqslant K$ for all $x\in\bR$.
Thus, the integral $\int^t_0 \int_{|h| \leqslant 1} \{f(X_{s-} + h) - f(X_{s-})\} \tilde{N}(ds,dh)$ with respect to the compensated Poisson random measure is a square integrable martingale.
}
\end{remark}

Then, we can rewrite Proposition \ref{prop2-1} to the following:
\begin{proposition}\label{prop2-2}
For each $f \in C^2_{1+,b}(\bR)$, the following formula holds:
\begin{align*}
f(X_t) - f(X_0)
&=\int^t_0 \int_{\bR_0} \{f(X_{s-} + h) - f(X_{s-})\} \tilde{N}(ds,dh)
\\&\quad+\int^t_0 \int_{\bR_0} \{f(X_{s}+h) - f(X_{s})-f'(X_s)h\} \nu_\al(dh)ds.
\end{align*}
\end{proposition}

\begin{proof}
It is sufficient to check
\begin{align*}
\int^t_0 \int_{|h| > 1} \left| f(X_s + h) - f(X_s) \right|\nu_\al(dh)ds < \infty.
\end{align*}
Since $f \in C^2_{1+,b}(\bR)$,
it follows from the mean value theorem that
\begin{align*}
\int_{|h| > 1}|f(x + h) - f(x)|\nu_\al(dh)
&=\int_{|h| > 1}\left|\int^1_0f'(x+\he h)hd\he\right|\nu_\al(dh)
\\&\leqslant K\int_{|h| > 1} |h|\nu_\al(dh)
=\frac{K(c_++c_-)}{\al-1}.
\end{align*}
where $K$ is a positive constant such that $|f'(x)|\leqslant K$ for all $x\in\bR$.
\end{proof}

\begin{remark}
\rm{
The integral $\int^t_0 \int_{|h|>1} \{f(X_{s-} + h) - f(X_{s-})\} \tilde{N}(ds,dh)$ with respect to the compensated Poisson random measure is a martingale by the same argument as the proof of Proposition \ref{prop2-2}.
}
\end{remark}

We consider the operator as follows:
\begin{align*}
\cL f(x) :=\int_{\bR_0} \{f(x+h) - f(x)-f'(x)h\} \nu_\al(dh)
\end{align*}
for all $f \in C^2_{1+,b}(\bR)$ and $x\in \bR$.
The operator $\cL$ coincides with the infinitesimal generator of $X$ on $\cS(\bR)$, see Sato \cite[Theorem 31.5]{Sat}.
Using the operator $\cL$, the It\^o formula (Proposition \ref{prop2-2}) can be rewritten as follows:
\begin{align*}
f(X_t) - f(X_0) &= \int^t_0 \int_{\bR_0} \{f(X_{s-} + h) - f(X_{s-})\} \tilde{N}(ds,dh)
+ \int^t_0 \cL f(X_s) ds.
\end{align*}

Using the Fourier transform of $f \in L^1(\bR)$ defined by
\begin{equation*}
\cF[f](u) :=\frac{1}{\sqrt{2\pi}} \int_\bR e^{-iux} f(x)dx \quad \text{for $u \in \bR$},
\end{equation*}
and the inverse Fourier transform of $f \in L^1(\bR)$ defined by
\begin{equation*}
\cF^{-1}[f](x) := \frac{1}{\sqrt{2\pi}} \int_\bR e^{iux} f(u)du \quad \text{for $x \in \bR$},
\end{equation*}
 the operator $\cL$ on $\cS(\bR)$ is also represented as follows: 
\begin{proposition}[{\cite[Theorem 3.3.3]{App}}]\label{prop2-3}
For each $f \in \cS(\bR)$ and $x \in \bR$,
\begin{equation*}
\cL f(x) = \cF^{-1}[\eta(u)\cF[f](u)](x).
\end{equation*}
\end{proposition}

We will use the following moments of stable processes.
\begin{lemma}[{\cite[Example 25.10]{Sat}}]\label{lem2-1}
Let $X$ be a stable process of index $\al \in (0,2)$ with the triplet $(b,0,\nu_\al)$.
Then, for $t\geqslant 0$,
\begin{equation*}
\bE|X_t|^\ga =
  \begin{cases}
  < \infty &\quad \text{if $0 <\ga < \al$}, \\
  = \infty &\quad \text{if $\ga \geqslant \al$}.
  \end{cases}
\end{equation*}
\end{lemma}

Moreover, we will use the following negative-order moments.
\begin{lemma}\label{lem2-2}
Let $X$ be a stable processes of index $\al \in (0,2)$ with the triplet $(b,0,\nu_\al)$.
Then, for all $t>0$ and $x\in\bR$,
\begin{equation*}
\bE[|X_t-x|^{-\ga}] \leqslant S(\al,\ga)t^{-\ga\slash\al} \quad \text{if $0<\ga<1$},
\end{equation*}
where the constant $S(\al,\ga)$ is independent of $x$.
\end{lemma}

\begin{proof}
By the monotone convergence theorem, we have
\begin{align*}
\lim_{\ep \downarrow 0}\bE\left[|X_t-x|^{-\ga}e^{-\ep |X_t-x|}\right] = \bE[|X_t-x|^{-\ga}].
\end{align*}
Since $|e^{t\eta(\cdot)}|=e^{-dt|\cdot|^\al} \in L^1(\bR)$ for $t>0$,
transition density $p_t$ is given by
\begin{align*}
p_t(y)=\frac{1}{2\pi}\int_\bR e^{-iuy}e^{t\eta(u)}du,
\end{align*}
for each $t>0$ and $y\in\bR$.
Thus, it follows from Fubini's theorem
that for all $\ep>0$,
\begin{align*}
\bE\left[|X_t-x|^{-\ga}e^{-\ep |X_t-x|} \right]
&=\int_\bR |y-x|^{-\ga}e^{-\ep|y-x|}p_t(y)dy
\\&=\frac{1}{2\pi}\int_\bR e^{t\eta(u)}\int_\bR |y-x|^{-\ga}e^{-\ep|y-x|}e^{-iuy}dydu
\\&=\frac{1}{2\pi}\int_\bR e^{t\eta(u)-iux}\int_\bR|z|^{-\ga}e^{-\ep|z|-iuz}dzdu,
\end{align*}
since $|\cdot|^{-\ga}e^{-\ep|\cdot|}\in L^1(\bR)$.

Now, we will make use of the following identity:
\begin{equation}\label{2-2}
\int^\infty_0 x^{\xi-1}e^{-xz}dx =\Ga(\xi)z^{-\xi}
\end{equation}
for all $\xi>0$ and $\Re z>0$, where $\Re z$ is the real part of $z$.
By $1-\ga>0$ and $\Re(\ep\pm iu)=\ep>0$, we have
\begin{align*}
\int_\bR|z|^{-\ga}e^{-\ep|z|-iuz}dz
&=\int^\infty_0 z^{-\ga}e^{-(\ep+iu)z}dz +\int^\infty_0 z^{-\ga}e^{-(\ep-iu)z}dz
\\&=\Ga(1-\ga)(\ep+iu)^{\ga-1} + \Ga(1-\ga)(\ep-iu)^{\ga-1}.
\end{align*}
We then have for $u\neq0$,
\begin{align*}
\lim_{\ep\downarrow0} \{ (\ep+iu)^{\ga-1}+(\ep-iu)^{\ga-1}\} &=(iu)^{\ga-1}+(-iu)^{\ga-1}
\\&=2|u|^{\ga-1}\cos\left(\frac{\pi(\ga-1)}{2}\right).
\end{align*}
By $-1<\ga-1<0$, we have
\begin{align*}
|\ep\pm iu|^{\ga-1}|e^{t\eta(u)-iux}|\leqslant |u|^{\ga-1}e^{-dt|u|^{\al}} \in L^1(\bR).
\end{align*}
Hence, it follows from the dominated convergence theorem that for all $t>0$ and $x\in\bR$,
\begin{align*}
&\lim_{\ep\downarrow 0}\bE\left[|X_t-x|^{-\ga}e^{-\ep |X_t-x|}\right]
\\&\quad=\frac{\Ga(1-\ga)}{2\pi}\lim_{\ep\downarrow 0}\int_\bR\{(\ep+iu)^{\ga-1}+(\ep-iu)^{\ga-1}\}e^{t\eta(u)-iux}du
\\&\quad=\frac{\Ga(1-\ga)}{\pi}\cos\left(\frac{\pi(\ga-1)}{2}\right)\int_\bR |u|^{\ga-1} e^{t\eta(u)-iux}du
\\&\quad\leqslant\frac{\Ga(1-\ga)}{\pi}\cos\left(\frac{\pi(\ga-1)}{2}\right)\int_\bR |u|^{\ga-1}e^{-dt|u|^{\al}}du
\\&\quad=\frac{\Ga(1-\ga)t^{-\ga\slash\al}}{\pi}\cos\left(\frac{\pi(\ga-1)}{2}\right)\int_\bR|v|^{\ga-1}e^{-d|v|^{\al}}dv
<\infty.
\end{align*}
The proof is now complete.
\end{proof}

In order to use the continuity about the law of $X_t$,
we check the following condition.
\begin{lemma}[{\cite[Theorem 27.4]{Sat}}]\label{lem2-3}
Let $Y$ be a L\'evy process with the triplet $(b,a,\nu)$.
The followings are equivalent.
\begin{flalign*}
\quad(i)&\quad \bP(Y_t=y)=0\quad\text{for all $y\in\bR$ and $t>0$}.&
\\ \quad(ii)&\quad a\neq0\quad\text{or}\quad\nu(\bR_0)=\infty.&
\end{flalign*}
\end{lemma}
Thus, by $\nu_\al(\bR_0)=\infty$, we have that the law of $X_t$ is continuous for all $t>0$.

\section{Main results}\label{sec3}
Before we establish the Tanaka formula for general strictly stable processes,
we need the following lemma.
In case of $\be =0$, that is, $X$ is a symmetric stable process with index $\al \in(1,2)$,
it is shown in Komatsu \cite{Kom} the following result:
\begin{lemma}\label{lem3-1}
Let 
\begin{equation*}
F(x) := D(\al)(1-\be\sgn(x))|x|^{\al-1}
\end{equation*}
where 
\begin{align*}
D(\al)=\frac{c(-\al)}
{d\left(1 + \be^2 \tan^2 (\pi \al \slash 2)\right)}.
\end{align*}
Then,
we have for all $\phi\in C^\infty_c(\bR)$ and $x\in\bR$,
\begin{align*}
\cL (F*\phi)(x)=\phi(x)
\end{align*}
where $F * \phi$ is given by  the convolution of $F$ and $\phi$:
\begin{equation*}
(F*\phi)(x) := \int_\bR F(y)\phi(x-y) dy.
\end{equation*}
\end{lemma}

\begin{proof}
Let $F_\phi := F * \phi$ for $\phi \in C^\infty_c(\bR)$.
We know $F_\phi \in C^\infty(\bR)$.
Since we have for all $n\in\bN$,
\begin{align*}
F^{(n)}_\phi(x)&=\int_\bR F(y)\phi^{(n)}(x-y)dy
=\int_\bR F'(y)\phi^{(n-1)}(x-y)dy,
\end{align*}
where
\begin{equation*}
F'(y)=(\al-1)D(\al)(\sgn(y)-\be)|y|^{\al-2},
\end{equation*}
we have $F^{(n)}_\phi$ is bounded for all $n\in\bN$.
Thus, we have $F_\phi \in C^\infty_{1+,b}(\bR)$.

Set $F_\ep(x) := F(x)e^{-\ep|x|}$ for $\ep > 0$
and
$F_{\ep,\phi} := F_\ep * \phi \in C^\infty(\bR)$.
Since $1+|x|^2\leqslant2(1+|x-y|^2)(1+|y|^2)$ for all $x,y\in\bR$,
we have for all $k,n\in\bN$,
\begin{align*}
&\sup_{x\in\bR}(1+|x|^2)^k |F^{(n)}_{\ep,\phi}(x)|
\\&\quad\leqslant 2^k\sup_{x\in\bR}\int_\bR(1+|x-y|^2)^k(1+|y|^2)^k  |F_\ep(y)\phi^{(n)}(x-y)|dy
\\&\quad\leqslant 2^{k+1}D(\al)\sup_{z\in\bR}(1+|z|^2)^k|\phi^{(n)}(z)|\int_\bR(1+|y|^2)^k|y|^{\al-1}e^{-\ep|y|}dy<\infty,
\end{align*}
since $\phi\in C^\infty_c(\bR)\subset\cS(\bR)$.
Thus, we have $F_{\ep,\phi}\in\cS(\bR)$.
By Proposition \ref{prop2-3},
we have
\begin{align*}
&\int_{\bR_0}\{F_{\ep,\phi}(x+h)-F_{\ep,\phi}(x)-F'_{\ep,\phi}(x)h\}dh
\\&\quad=\cF^{-1}[\eta(u)\cF[F_{\ep,\phi}](u)](x)
\\&\quad=\sqrt{2\pi}\cF^{-1}[\eta(u)\cF[F_{\ep}](u)\cF[\phi](u)](x).
\end{align*}

Since $|x| e^{-|x|}\leqslant 1$ for all $x\in\bR$, we have
\begin{align*}
|F'_\ep(x)|
&=\biggl|(\al-1)D(\al)(\sgn(x)-\be)|x|^{\al-2}e^{-\ep|x|}
\\&\quad-\ep D(\al)(\sgn(x)-\be)|x|^{\al-1}e^{-\ep|x|}\biggl|
\\&\leqslant 2D(\al)|x|^{\al-2}e^{-\ep|x|}+2\ep D(\al)|x|^{\al-1}e^{-\ep|x|}
\\&\leqslant 4D(\al)|x|^{\al-2}.
\end{align*}
By using integration by parts, it follows from Fubini's theorem that
\begin{align*}
&|F_{\ep,\phi}(x+h)-F_{\ep,\phi}(x)-F'_{\ep,\phi}(x)h|
\\&\quad=\left|\int_\bR F_{\ep}(y)\{\phi(x+h-y)-\phi(x-y)-\phi'(x-y)h\} dy\right|
\\&\quad=\left|\int_\bR F'_\ep(y)h\int^1_0\{\phi(x+\he h-y)-\phi(x-y)\}d\he dy\right|
\\&\quad=\left|\int^1_0\int_\bR F'_\ep(y)h\{\phi(x+\he h-y)-\phi(x-y)\}dyd\he\right|
\\&\quad\leqslant 2C_1(\al)|h|,
\end{align*}
where the constant $C_1(\al)$ is given by
\begin{align*}
C_1(\al) = 4D(\al)\sup_{x\in\bR}\int_\bR|y|^{\al-2}|\phi(x-y)|dy.
\end{align*}
Similarly, it follows that
\begin{align*}
&|F_{\ep,\phi}(x+h)-F_{\ep,\phi}(x)-F'_{\ep,\phi}(x)h|
\\&\quad=\left|\int_\bR F'_{\ep}(y)h^2\int^1_0\phi'(x+\he h-y)(1-\he)d\he dy\right|
\\&\quad=\left|\int^1_0\int_\bR F'_{\ep}(y)h^2\phi'(x+\he h-y)(1-\he)dyd\he\right|
\\&\quad\leqslant \frac{C_2(\al)|h|^2}{2},
\end{align*}
where the constant $C_2(\al)$ is given by
\begin{align*}
C_2(\al) = 4D(\al)\sup_{x\in\bR}\int_\bR|y|^{\al-2}|\phi'(x-y)|dy.
\end{align*}
By $1<\al<2$,
we know
\begin{align*}
\int_{\bR_0}(|h|^2\wedge|h|)\nu_\al(dh)=\frac{c_++c_-}{2-\al}+\frac{c_++c_-}{\al-1}<\infty.
\end{align*}
Hence, 
it follows from the dominated convergence theorem that
\begin{align*}
&\lim_{\ep\downarrow0}\sqrt{2\pi}\cF^{-1}[\eta(u)\cF[F_{\ep}](u)\cF[\phi](u)](x)
\\&\quad=\lim_{\ep\downarrow0}\int_{\bR_0}\{F_{\ep,\phi}(x+h)-F_{\ep,\phi}(x)-F'_{\ep,\phi}(x)h\}dh
\\&\quad=\int_{\bR_0} \{F_{\phi}(x+h)-F_{\phi}(x)-F'_{\phi}(x)h\}dh.
\end{align*}

Now, by using \eqref{2-2},
we have
\begin{align*}
\cF[|x|^{\al-1}e^{-\ep|x|}](u)
=\frac{\Ga(\al)}{\sqrt{2\pi}}\left( (\ep+iu)^{-\al}+(\ep-iu)^{-\al}\right),
\end{align*}
and
\begin{align*}
\cF[|x|^{\al-1}\sgn(x)e^{-\ep|x|}](u)
=\frac{\Ga(\al)}{\sqrt{2\pi}}\left( (\ep+iu)^{-\al}-(\ep-iu)^{-\al}\right).
\end{align*}
We then have for $u\neq0$,
\begin{align*}
\lim_{\ep\downarrow0}\cF[|x|^{\al-1}e^{-\ep|x|}](u)
&=\frac{\Ga(\al)}{\sqrt{2\pi}}|u|^{-\al}\left(i^{-\al}+(-i)^{-\al}\right)
\\&=\frac{2\Ga(\al)}{\sqrt{2\pi}}|u|^{-\al}\cos\left(\frac{\pi\al}{2}\right),
\end{align*}
and
\begin{align*}
\lim_{\ep\downarrow0}\cF[|x|^{\al-1}e^{-\ep|x|}\sgn(x)](u)
&=\frac{\Ga(\al)}{\sqrt{2\pi}}|u|^{-\al}\sgn(u)\left(i^{-\al}-(-i)^{-\al}\right)
\\&=-i\frac{2\Ga(\al)}{\sqrt{2\pi}}|u|^{-\al}\sgn(u)\sin\left(\frac{\pi\al}{2}\right).
\end{align*}
Thus, we have
\begin{align*}
&\lim_{\ep\downarrow0}\eta(u)\cF[F_{\ep}](u)
\\&\quad=\left(-c(-\al)\frac{d|u|^{\al}(1-i\be\sgn(u)\tan(\pi\al\slash2))}{d(1+\be^2\tan^2(\pi\al\slash2))}\right)
\\&\qquad\times \lim_{\ep\downarrow0}\left(\cF[|x|^{\al-1}e^{-\ep|x|}](u)-\be\cF[|x|^{\al-1}e^{-\ep|x|}\sgn(x)](u)\right)
\\&\quad= \left(-\frac{c(-\al)|u|^{\al}}{1+i\be\sgn(u)\tan(\pi\al\slash2)}\right)
\\&\qquad\times \frac{2\Ga(\al)}{\sqrt{2\pi}}|u|^{-\al}\left(\cos\left(\frac{\pi\al}{2}\right)+i\be\sgn(u)\sin\left(\frac{\pi\al}{2}\right)\right)
\\&\quad=-\frac{2\Ga(\al)c(-\al)}{\sqrt{2\pi}}\cos\left(\frac{\pi\al}{2}\right)=\frac{1}{\sqrt{2\pi}}.
\end{align*}
By $\phi\in C^\infty_c(\bR)$,
we have
\begin{align*}
\left|\frac{u}{\ep\pm iu}\right|^{\al}\cF[\phi](u)\leqslant\cF[\phi](u)\in C^\infty_c(\bR).
\end{align*}
Hence, it follows from the dominated convergence theorem that
\begin{align*}
&\lim_{\ep\downarrow0}\sqrt{2\pi}\cF^{-1}[\eta(u)\cF[F_{\ep}](u)\cF[\phi](u)](x)
\\&\quad=\sqrt{2\pi}\cF^{-1}\left[\lim_{\ep\downarrow0}\eta(u)\cF[F_{\ep}](u)\cF[\phi](u)\right](x)
\\&\quad=\cF^{-1}[\cF[\phi](u)](x) =\phi(x).
\end{align*}
The proof is now complete.
\end{proof}

\begin{remark}
\rm{
A strictly stable process with index $\al \in (0,1)$ is characterized by $(b'_\al,0,\nu_\al)$
where $b'_\al=\int_{|h|\leqslant1}h\nu_\al(dh)=(c_+-c_-)\slash(1-\al)$.
In this case, the operator corresponding to $\cL$ can be represented as
\begin{align*}
\cL' f(x)=\int_{\bR_0}\{f(x+h)-f(x)\}\nu_\al(dh)
\end{align*}
for $f\in C^1_{1+,b}(\bR)$ which is bounded.
Then, the Lemma \ref{lem3-1} holds for $\cL'$.
But the local time does not exist if $\al \in (0,1)$.
Indeed, in {\cite[Theorem V.1]{Bert}}, the condition
\begin{equation}\label{3-1}
\int_\bR \Re\left(\frac{1}{1-\eta(u)}\right) du < \infty
\end{equation}
holds if and only if the occupation time measure $\mu_t$ is absolutely continuous with respect to the Lebesgue measure such that its density is in $L^2(da \otimes d\bP)$.
Moreover, if \eqref{3-1} fails, then the measure $\mu_t$ is singular for every $t>0$, almost surely.
In case of $\al \in (0,1)$, since \eqref{3-1} does not hold, the measure $\mu_t$ is singular.
}
\end{remark}

We introduce a mollifier as follows:
\begin{definition}
\rm{
The function $\rho$ is a mollifier if,
the function $\rho$ on $\bR$ satisfies
\begin{flalign*}
\quad(i)&\quad\rho(x)\geqslant0 \quad \text{for all $x\in\bR$}.&
(ii)&\quad\rho\in C^\infty_c(\bR).&
\\ \quad(iii)&\quad\supp \rho =[-1,1].&
(iv)&\quad\int_\bR\rho(x)dx=1.&
\end{flalign*}
Furthermore, we define $\rho_n(x):=n\rho(nx)$ for all $n\in\bN$.
A family of functions $(\rho_n)_{n\in\bN}$ satisfies that
$\rho_n \to \de_0$ as $n\to\infty$ where $\de_0$ is the Dirac delta function, in the sense of Schwartz distributions, that means
\begin{align*}
\left|\int_\bR\rho_n(x)\phi(x)dx-\phi(0)\right| \to 0 \quad \text{as $n\to\infty$}
\end{align*}
for all $\phi\in\cS(\bR)$.
}
\end{definition}

Now let us state our main theorem which we call the Tanaka formula for arbitrary strictly stable processes with index $\al \in (1,2)$.
\begin{theorem}\label{thm3-1}
Let $F$ be the same as in Lemma \ref{lem3-1}.
Then, 
there exists the local time $L = \{L^a_t: a \in \bR, t\geqslant 0\}$ of $X$
which is given by
\begin{equation*}
L^a_t=F(X_t-a)-F(X_0-a)-M^a_t\quad\text{for $a\in\bR,t\geqslant0$},
\end{equation*}
where the process $(M^a_t)_{t\geqslant0}$ given by
\begin{equation*}
M^a_t=\int ^t_0\int_{\bR_0}\{F(X_{s-}-a+h)-F(X_{s-}-a)\}\tilde{N}(ds,dh)
\end{equation*}
is a square integrable martingale.
\end{theorem}

\begin{remark}
\rm{
The local time can be also represented by
\begin{align*}
L^a_t=\lim_{n\to\infty}\int^t_0\rho_n(X_s-a)ds\quad\text{in $L^2(d\bP)$},
\end{align*}
where $(\rho_n)_{n\in\bN}$ is given by $\rho_n(x)=n\rho(nx)$ for all $n\in\bN$
with a mollifier $\rho$.
}
\end{remark}

\begin{proof}[Proof of Theorem {\ref{thm3-1}}]
Let $F_n = F * \rho_n$ for all $n\in\bN$.
By the argument in the proof of Lemma \ref{lem3-1},
we have $F_n \in C^{\infty}_{1+, b}(\bR)$.
By the It\^o formula (Proposition \ref{prop2-2}),
we have for all $t\geqslant0$ and $a\in\bR$,
\begin{equation*}
F_n(X_t-a)-F_n(X_0-a) = M^{a,n}_t + V^{a,n}_t,
\end{equation*}
where 
\begin{equation*}
M^{a,n}_t = \int^t_0 \int_{\bR_0} \{F_n(X_{s-} - a + h) - F_n(X_{s-} - a)\} \tilde{N}(ds,dh)
\end{equation*}
and
\begin{equation*}
V^{a,n}_t = \int^t_0 \cL F_n(X_s - a) ds.
\end{equation*}

First we will show that $F_n(X_t-a), F(X_t-a) \in L^2(d\bP)$ and $F_n(X_t-a)\to F(X_t-a)$ in $L^2(d\bP)$ as $n\to\infty$.
Using the following inequality:
\begin{align*}
|x+y|^{\al-1} \leqslant |x|^{\al-1} + |y|^{\al-1}\quad\text{for all $x,y \in \bR$},
\end{align*}
we have for $x \in \bR$,
\begin{align*}
0 &\leqslant F_n(x)=\int_\bR F(x-y)\rho_n(y)dy
\\&\leqslant \int^{1\slash n}_{-1\slash n} 2D(\al)\left(|x|^{\al-1}+|y|^{\al-1}\right)\rho_n(y)dy
\leqslant 2D(\al)\left(|x|^{\al-1} + 1\right).
\end{align*}
By Lemma \ref{2-1},
we have
\begin{align*}
\bE[|F_n(X_t-a)|^2]&\leqslant4D(\al)^2\bE[(|X_t-a|^{\al-1}+1)^2]
\\&\leqslant12D(\al)^2\left(\bE[|X_t|^{2\al-2}]+|a|^{2\al-2} +1\right)<\infty,
\end{align*}
since $0<2\al-2<\al$.
Similarly,
we have
\begin{align*}
\bE[|F(X_t-a)|^2]&\leqslant4D(\al)^2\bE[(|X_t|^{\al-1}+|a|^{\al-1})^2]
\\&\leqslant8D(\al)^2\left(\bE[|X_t|^{2\al-2}]+|a|^{2\al-2}\right)<\infty.
\end{align*}
Hence, it follows from the dominated convergence theorem that
\begin{align}\label{3-2}
&\lim_{n\to\infty}\bE[|F_n(X_t-a)-F(X_t-a)|^2] \notag
\\&\quad=\bE\left[\lim_{n\to\infty}|F_n(X_t-a)-F(X_t-a)|^2\right] =0.
\end{align}

Next, we will show that $M^{a,n}_t$ and $M^a_t$ are square integrable martingales,
and $M^{a,n}_t\to M^a_t$ in $L^2(d\bP)$ as $n\to\infty$.
Since $\left| |x+1|^{\al-1}-1\right|\leqslant|x|$ for all $x\in\bR$,
we have for all $0\leqslant\ep\leqslant2-\al$,
\begin{align*}
\left| |x+1|^{\al-1}-1\right|
&\leqslant |x|^{(\al+\ep)\slash2}\wedge|x|^{\al-1},
\\ \left| |x+1|^{\al-1}\sgn(x+1)-1\right|
&\leqslant 3\left(|x|^{(\al+\ep)\slash2}\wedge |x|^{\al-1}\right),
\end{align*}
by $\al-1<(\al+\ep)\slash2\leqslant1$.
Thus, we have for $x\neq0$ and $y\in\bR$,
\begin{align*}
&|F(x+y)-F(x)|^2
\\&\quad=|D(\al)\{|x+y|^{\al-1}-\be|x+y|^{\al-1}\sgn(x+y)\}
\\&\qquad-D(\al)\{|x|^{\al-1}-\be|x|^{\al-1}\sgn(x)\}|^2
\\&\quad\leqslant 2D(\al)^2\left| |x+y|^{\al-1}-|x|^{\al-1}\right|^2
\\&\qquad+2D(\al)^2\left| |x+y|^{\al-1}\sgn(x+y)-|x|^{\al-1}\sgn(x)\right|^2
\\&\quad=2D(\al)^2|x|^{2\al-2}\left| \left|1+\frac{y}{x}\right|^{\al-1}-1\right|^2
\\&\qquad+2D(\al)^2|x|^{2\al-2}\left| \left|1+\frac{y}{x}\right|^{\al-1}\sgn\left(1+\frac{y}{x}\right)-1\right|^2
\\&\quad\leqslant8D(\al)^2\left(|x|^{\al-\ep-2}|y|^{\al+\ep}\wedge|y|^{2\al-2}\right)
\end{align*}
Now, choose $\ep_0$ such that $0<\ep_0<(\al-1)\wedge(2-\al)$.
Since the law of $X_t$ is continuous for all $t>0$ by Lemma \ref{lem2-3},
it follows from the Cauchy--Schwarz inequality, Fubini's theorem and Lemma \ref{lem2-2} that for all $s>0$,
\begin{align*}
&\bE\left[|F_n(X_s+h-a)-F_n(X_s-a)|^2\right]
\\&\quad\leqslant\bE\left[\int_\bR\rho_n(y)|F(X_s+h-a-y)-F(X_s-a-y)|^2dy\right]
\\&\quad\leqslant8D(\al)^2\int_\bR\rho_n(y)|h|^{2\al-2}\textbf{1}_{\{|h|>1\}}dy
\\&\qquad+8D(\al)^2\int_\bR\rho_n(y)\bE\left[|X_s-a-y|^{\al-\ep_0-2}|h|^{\al+\ep_0}\right]\textbf{1}_{\{|h|\leqslant1\}}dy
\\&\quad\leqslant8D(\al)^2|h|^{2\al-2}\textbf{1}_{\{|h|>1\}}
\\&\qquad+8D(\al)^2S(\al,-\al+\ep_0+2)s^{(\al-\ep_0-2)\slash\al}|h|^{\al+\ep_0}\textbf{1}_{\{|h|\leqslant1\}},
\end{align*}
by $0<-\al+\ep_0+2<1$.
Similarly,
it follows that for all $s>0$,
\begin{align*}
&\bE\left[|F(X_s+h-a)-F(X_s-a)|^2\right]
\\&\quad\leqslant8D(\al)^2|h|^{2\al-2}\textbf{1}_{\{|h|>1\}}
\\&\qquad+8D(\al)^2\bE\left[|X_s-a|^{\al-\ep_0-2}|h|^{\al+\ep_0}\right]\textbf{1}_{\{|h|\leqslant1\}}
\\&\quad\leqslant8D(\al)^2|h|^{2\al-2}\textbf{1}_{\{|h|>1\}}
\\&\qquad+8D(\al)^2 S(\al,2+\ep_0-\al) s^{(\al-\ep_0-2)\slash\al} |h|^{\al+\ep_0}\textbf{1}_{\{|h|\leqslant1\}}.
\end{align*}
By $0<\ep_0<\al-1$ and $1<\al<2$,
we know
\begin{align*}
\int_{\bR_0}\left(|h|^{\al+\ep_0}\wedge|h|^{2\al-2} \right)\nu_\al(dh)=\frac{c_++c_-}{\ep_0}+\frac{c_++c_-}{2-\al}<\infty,
\end{align*}
and
\begin{align*}
\int^t_0 s^{(\al-\ep_0-2)\slash\al}ds=\frac{\al}{2\al-\ep_0-2}t^{(2\al-\ep_0-2)\slash\al}<\infty.
\end{align*}
Hence, it follows that $M^{a,n}_t$ and $M^a_t$ are square integrable martingales,
and from the dominated convergence theorem and  \eqref{3-2} that
\begin{align}\label{3-3}
&\lim_{n\to\infty}\bE[|M^{a,n}_t-M^a_t|^2]  \notag
\\&\quad=\lim_{n\to\infty}\int^t_0\int_{\bR_0}\bE[|F_n(X_s+h-a)-F_n(X_s-a) \notag
\\&\qquad-\{F(X_s+h-a)-F(X_s-a)\}|^2]\nu_\al(dh)ds \notag
\\&\quad=\int^t_0\int_{\bR_0}\lim_{n\to\infty}\bE[|F_n(X_s+h-a)-F_n(X_s-a) \notag
\\&\qquad-\{F(X_s+h-a)-F(X_s-a)\}|^2]\nu_\al(dh)ds\notag
\\&\quad=0. 
\end{align}

Finally, we will show that $L^a_t:=F(X_t-a)-F(X_0-a)-M^a_t$ is the local time of $X$.
It is sufficient to show that an occupation time formula holds for $g\in C_c(\bR)$. 
By Lemma \ref{lem3-1} and Fubini's theorem, we have for a.e.-$\om$,
\begin{align*}
\int_\bR g(a)V^{a,n}_t(\om)da
&=\int_\bR g(a)\int^t_0\rho_n(X_s(\om)-a)dsda
\\&=\int^t_0 (g*\rho_n)(X_s(\om))ds.
\end{align*}
Thus, we will show that
\begin{align*}
\lim_{n\to\infty}\int_\bR g(a)V^{a,n}_tda=\int_\bR g(a)L^a_tda \quad\text{in $L^2(d\bP)$},
\end{align*}
and
\begin{align*}
\lim_{n\to\infty}\int^t_0 (g*\rho_n)(X_s)ds=\int^t_0g(X_s)ds\quad\text{in $L^2(d\bP)$}.
\end{align*}
By using the Cauchy--Schwarz inequality and Fubini's theorem,
it follows from the dominated convergence theorem, \eqref{3-2} and \eqref{3-3} that for $g\in C_c(\bR)$,
\begin{align*}
&\bE\left[\left|\int_\bR g(a)V^{a,n}_tda-\int_\bR f(a)L^a_tda\right|^2\right]
\\&\quad\leqslant\bE\left[\int_\bR |g(a)|da\int_\bR|g(a)||V^{a,n}_t-L^a_t|^2da\right]
\\&\quad\leqslant\int_\bR |g(a)|da\int_\bR|g(a)|\bE\left[|F_n(X_t-a)-F(X_t-a)|^2\right] da
\\&\qquad+\int_\bR |g(a)|da\int_\bR|g(a)|\bE\left[|F_n(X_0-a)-F(X_0-a)|^2\right] da
\\&\qquad+\int_\bR |g(a)|da\int_\bR|g(a)|\bE\left[|M^{a,n}_t-M^a_t|^2\right] da
\\&\quad\to 0 \quad \text{as $n\to\infty$}.
\end{align*}
By using the Cauchy--Schwarz inequality and Fubini's theorem,
it follows from the dominated convergence theorem that for $g\in C_c(\bR)$,
\begin{align*}
&\bE\left[\left|\int^t_0 (g*\rho_n)(X_s)ds-\int^t_0g(X_s)ds\right|^2\right]
\\&\quad\leqslant t\bE\left[\int^t_0\left|\int_\bR \rho_n(y)\{g(X_s-y)-g(X_s)\}dy\right|^2ds\right]
\\&\quad\leqslant t\bE\left[\int^t_0\int_\bR \rho_n(y)|g(X_s-y)-g(X_s)|^2dyds\right]
\\&\quad=t\int^t_0\int_\bR n\rho(ny)\bE\left[|g(X_s-y)-g(X_s)|^2\right]dyds
\\&\quad= t\int^t_0\int_\bR\rho(z)\bE\left[|g(X_s-n^{-1}z)-g(X_s)|^2\right]dzds
\\&\quad\to 0 \quad \text{as $n\to\infty$}.
\end{align*}
The proof is now complete.
\end{proof}

\section*{Acknowledgements}
I am very grateful to Professor Yoshiki Otobe of Shinshu University for his kind guidance and valuable discussion.
I also thank Professor Kouji Yano of Kyoto University and Professor Atsushi Takeuchi of Osaka City University for their valuable advice.


\begin{thebibliography}{99}
\bibitem{App}
D. Applebaum,
L\'evy processes and stochastic calculus.
Second edition.
Cambridge Studies in Advanced Mathematics, 116,
Cambridge University Press, Cambridge, 2009.

\bibitem{Bar}
M. T. Barlow,
{\it Necessary and sufficient conditions for the continuity of local time of L\'evy processes},
Ann. Probab. 16 (1988), no. 4, pp. 1389--1427.

\bibitem{Berm}
S. M. Berman,
{\it Local times and sample function properties of stationary Gaussian processes}, 
Trans. Amer. Math. Soc. 137 (1969), pp. 277--299.

\bibitem{Bert}
J. Bertoin,
L\'evy processes.
Cambridge Tracts in Mathematics, 121,
Cambridge University Press, Cambridge, 1996.

\bibitem{Boy}
E. S. Boylan,
{\it Local times for a class of Markoff processes}, 
Illinois J. Math. 8 (1964), pp. 19--39.

\bibitem{Eng}
H. J. Engelbert, V. P. Kurenok and A. Zalinescu,
{\it On existence and uniqueness of reflected solutions of stochastic equations driven by symmetric stable processes},
From Stochastic Calculus to Mathematical Finance, 
Springer Berlin Heidelberg, (2006), pp. 227--248.

\bibitem{Ike}
N. Ikeda and S. Watanabe,
Stochastic differential equations and diffusion processes. 
Second edition.
North-Holland Mathematical Library, 24,
North-Holland Publishing Co., Amsterdam; Kodansha, Ltd., Tokyo, 1989.

\bibitem{Kom}
T. Komatsu,
{\it On the pathwise uniqueness of solutions of one-dimensional stochastic differential equations of jump type}, 
Proc. Japan Acad. Ser. A Math. Sci. 58 (1982), no. 8, pp. 353--356.

\bibitem{Sal}
P. Salminen and M. Yor,
{\it Tanaka formula for symmetric L\'evy processes},
S\'eminaire de Probabilit\'es XL,  
Lecture Notes in Math., 1899, 
Springer, Berlin, (2007), pp. 265--285.

\bibitem{Sat}
K. Sato,
L\'evy processes and infinitely divisible distributions.
Cambridge Studies in Advanced Mathematics, 68,
Cambridge University Press, Cambridge, 1999.

\bibitem{Tro}
H. F. Trotter,
{\it A property of Brownian motion paths}, 
Illinois J. Math. 2 (1958), pp. 425--433.

\bibitem{Wat}
S. Watanabe,
{\it On stable processes with boundary conditions},
J. Math. Soc. Japan. 14.2 (1962), pp. 170--198.

\bibitem{Yam}
K. Yamada,
{\it Fractional derivatives of local times of $ \alpha $-stable Levy processes as the limits of occupation time problems},
Limit theorems in probability and statistics, Vol. II (Balatonlelle, 1999),
J\'anos Bolyai Math. Soc., Budapest, (2002), pp. 553--573.

\bibitem{Yan}
K. Yano,
{\it On harmonic function for the killed process upon hitting zero of asymmetric L\'evy processes},
J. Math-for-Ind. 5A (2013), pp. 17--24.
\end{thebibliography}
\end{document}